\documentclass[11pt]{amsart}
\theoremstyle{plain}
\newtheorem{thm}{Theorem}[section]
\newtheorem{theorem}[thm]{Theorem}

\newtheorem{lemma}[thm]{Lemma}
\newtheorem{corollary}[thm]{Corollary}
\newtheorem{proposition}[thm]{Proposition}
\theoremstyle{definition}

\newtheorem{question}[thm]{Question}

\numberwithin{equation}{section}

\title [Smooth quartic K3 surfaces]{Smooth quartic K3 surfaces 
and Cremona transformations, II}

\author{Keiji Oguiso}

\address{Keiji Oguiso, Department of Mathematics, Osaka University\\
Toyonaka 560-0043 Osaka, Japan and  Korea Institute for Advanced Study, Hoegiro 87, Seoul, 130-722, Korea} \email{oguiso@math.sci.osaka-u.ac.jp}

\thanks{supported by JSPS Gran-in-Aid (B) No 22340009, JSPS Grant-in-Aid (S), No 22224001, and by KIAS Scholar Program}

\begin{document}

\maketitle

\begin{abstract}
This is a continuation of [Og12], concerning automorphisms of smooth quartic K3 surfaces and birational automorphisms of ambient projective three spaces. 
\end{abstract}

\section{Introduction}

Thoughout this note we work over $\mathbf C$. This note is a continuation of our previous paper \cite{Og12}. Our main results are Theorems (\ref{main1}), (\ref{main2}) below. In \cite{Og12}, being based on a result of Takahashi (\cite{Ta98}, see also Theorem (\ref{takahashi})), we gave the first {\it negative} answer to the following long standing question of Gizatullin (\cite{Do11}): 
\begin{question}\label{gizatullin}
Let $S \subset {\mathbf P}^3$ be a smooth quartic K3 surface and $g \in {\rm Aut}\, (S)$ as abstarct variety of $X$. Is $g$ derived from ${\rm Bir}\, ({\mathbf P}^3)$, the group of birational automorphisms of ${\mathbf P}^3$, in the sense 
that there is $\tilde{g} \in {\rm Bir}\, ({\mathbf P}^3)$ such that 
$\tilde{g}_* S = S$ and $\tilde{g} \vert S = g$?
\end{question}
The counter example in \cite{Og12} is of Picard number $2$ and $g$ is of infinite order. Construction there is lattice theoretic and in this sense implicit. Here and hereafter, a {\it K3 surface} is a smooth projective simply connected surface $S$ with nowhere vanishing global holomorphic $2$-form $\sigma_S$. Recall that ${\rm Aut}\, (S)$ can be an infinite group and can admit infinitely many different embeddings into ${\mathbf P}^3$, up to ${\rm Aut}\, ({\mathbf P}^3)$, or equivalently, infinitely many different linearly equivalent classes of very ample divisors $D$ with $(D^2)_S = 4$. The Fermat quartic surface is such an example. Keeping these in mind, one can naturally ask more interesting questions now than the original question. For instance: 

\begin{question}\label{gizatullin2}
Can one describe at least one counter example $(S, g)$ of Question (\ref{gizatullin}) explicitly in terms of homogeneous coordinates of ${\mathbf P}^3$? 
\end{question}

\begin{question}\label{gizatullin3}
Is there a smooth quartic K3 surface $S \subset {\mathbf P}^3$ such that 
${\rm Aut}\, (S)$ is enough complicated but every element of ${\rm Aut}\, (S)$ is derived from  ${\rm Bir}\, ({\mathbf P}^3)$?  
\end{question}

\begin{question}\label{gizatullin4}
Is every automorphism of the Fermat quartic surface $S \subset {\mathbf P}^3$ derived from ${\rm Bir}\, ({\mathbf P}^3)$?
\end{question}

\begin{question}\label{gizatullin5}
Is every automorphism of finite order of any smooth quartic surface $S \subset {\mathbf P}^3$ derived from ${\rm Bir}\, ({\mathbf P}^3)$?
\end{question}

\begin{question}\label{gizatullin6}
Are there a smooth K3 surface $S$ and $g \in {\rm Aut}\, (S)$ such that $g$ 
is not derived from ${\rm Bir}\, ({\mathbf P}^3)$ in any embedding $\Phi : S \to {\mathbf P}^3$?
\end{question}

The aim of this note is to give affrimative answers to Questions (\ref{gizatullin3}), (\ref{gizatullin6}) in stronger forms as well as an example supporting Question (\ref{gizatullin5}). Questions (\ref{gizatullin2}), (\ref{gizatullin4}) are completely open and Question (\ref{gizatullin5}) is also mostly open  
(see however \cite{Pr12}). 

Our main results are the following:

\begin{theorem}\label{main1}
There is a smooth quartic K3 surface $S \subset {\mathbf P}^3$ of Picard number $3$ such that
\begin{enumerate}
\item ${\rm Aut}\, (S) \simeq {\mathbf Z}_2 * {\mathbf Z}_2 * {\mathbf Z}_2$,  
the free product of three cyclic groups of order $2$. 
\item Every element of ${\rm Aut}\, (S)$ is derived from  ${\rm Bir}\, ({\mathbf P}^3)$ but no element other than $id_S$ is derived from ${\rm Aut}\, ({\mathbf P}^3) = {\rm PGL}\, (4, {\mathbf C})$. 
\item $S$ admits infinitely many different embeddings $S \to {\mathbf P}^3$, up to ${\rm Aut}\, ({\mathbf P}^3)$, and ${\rm Aut}\, (S)$ is derived from ${\rm Bir}\, ({\mathbf P}^3)$ in any embedding $S \to {\mathbf P}^3$. 
\end{enumerate}
\end{theorem}

\begin{theorem}\label{main2}
There is a smooth K3 surface $S$ of Picard number $2$ such that
\begin{enumerate}
\item ${\rm Aut}\, (S) \simeq {\mathbf Z}$. 
\item $S$ admits infinitely many different embeddings $S \to {\mathbf P}^3$ up to ${\rm Aut}\, ({\mathbf P}^3)$.
\item No element ${\rm Aut}\, (S)$ other than 
$id_S$ is derived from ${\rm Bir}\, ({\mathbf P}^3)$ 
in any embeddings of $S$ into ${\mathbf P}^3$. 
\end{enumerate}
\end{theorem}

We construct a K3 surface $S$ in Theorem (\ref{main2}) by deforming the Fermat quartic K3 surface suitably. This will be done in Section 3 after recalling some preliminary results in Section 2. Our proof is involved but fairly concrete. Note that $S$ admits infinitely many involutions but no automorphism of finite order $\ge 3$ (Theorem (\ref{main1})(1)). Theorem (\ref{main1})(3), in particular, says that all finite order automorphisms of $S$ are derived from ${\rm Bir}\, ({\mathbf P}^3)$ but none from ${\rm Aut}\, ({\mathbf P}^3)$ except $id_S$ (cf. Question (\ref{gizatullin5})).  

As one of relevant important results, it is worth mentioning here that the Kummer surface ${\rm Km}({\rm Jac}\, (C))$ of the Jacobian of a generic curve $C$ of genus $2$ is birational to a quartic surface $Q$ with $16$ nodes and ${\rm Aut}\, ({\rm Km}({\rm Jac}\, (C)))$ is derived from ${\rm Bir}\, ({\mathbf P}^3)$ via $Q$ (\cite{Ke97}, \cite{Ko98}). 

Unlike Theorem(\ref{main1})(3), Theorem (\ref{main2}) gives an affrimative answer to Question (\ref{gizatullin6}). This is originally asked by Doctor Sergey Galkin to me at Kinosaki conference (2011, October). Theorem (\ref{main2}) is 
proved in Section 4 again being based on the following special case of a more general result of Takahashi (\cite[Theorem 2.3, Remark 2.4]{Ta98}):

\begin{theorem}\label{takahashi}
Let $H$ be the hyperplane class of ${\mathbf P}^3$ and let $S \subset {\mathbf P}^3$ be a smooth quartic surface such that any effective curve $C$ on $S$ with ${\rm deg}\, C := (C.H)_{{\mathbf P}^3} < 16$ is of the form $C = T \vert S$ for some hypersurface $T$ of ${\mathbf P}^3$. Then, for any $g \in {\rm Bir}\,({\mathbf P}^3) \setminus {\rm Aut}\,({\mathbf P}^3)$, $K_{{\mathbf P}^3} + g_*S$ is ample. In particular, $g_*S \not= S$. 
\end{theorem}

{\bf Notation.} Throughout this note $L_K = L \otimes_{{\mathbf Z}} K$ for ${\mathbf Z}$-module $L$ and ${\mathbf Z}$-algebra $K$. For a K3 surface $S$, ${\rm NS}\, (S)$ is the N\'eron-Severi lattice. This is naturally isomorphic to ${\rm Pic}\, (S)$ and is a sublattice of $H^2(S, {\mathbf Z})$ with natural intersection form $(*.**)_S$. The orthogonal complement $T(S) := {\rm NS}\, (S)^{\perp}_{H^2(S, {\mathbf Z})}$ of ${\rm NS}\, (S)$ in $H^2(S, {\mathbf Z})$ is the transcendental lattice. $T(S)$ is the minimal primitive sublattice of $H^2(S, {\mathbf Z})$ such that $H^0(S, \Omega_S^2) = {\mathbf C}\sigma_S \subset T_{{\mathbf C}}$. The dual lattice ${\rm NS}\, (S)^*$ of ${\rm NS}\, (S)$ is $\{x \in {\rm NS}\, (S)_{{\mathbf Q}}\, \vert \, (x.{\rm NS}\,(S))_S \subset {\mathbf Z}\}$. The dual lattice $T(S)^{*}$ is defined similarly from $T(S)$. We have ${\rm NS}\,(S) \subset {\rm NS}\,(S)^*$, $T(S) \subset T(S)^*$ and there is a natural isomorphism ${\rm NS}\, (S)^*/{\rm NS}\, (S) \simeq T(S)^*/T(S)$ compatible with the natural action of ${\rm Aut}\, (S)$ (\cite[Proposition 1.6.1]{Ni79}). Similarly $L^{*}/L \simeq (L^{\perp})^*/L^{\perp}$ for non-degenerate primitive sublattice $L$. This isomorphism is compatible with the natural action of ${\rm Aut}\, (S, L)$, the subgroup consisting of elements $g \in {\rm Aut}\, (S)$ such that $g^*(L) = L$. The positive cone $P(S)$ is the connected component of $\{x \in {\rm NS}\, (S)\, \vert\, (x^2)_S > 0\}$ containing ample classes. The effective nef cone $\overline{{\rm Amp}}^e\,(S)$ is the convex cone in ${\rm NS}\, (S)_{\mathbf R}$ generated by the effective nef classes. The nef cone $\overline{{\rm Amp}}\,(S)$ is the closure of $\overline{{\rm Amp}}^e\,(S)$ in the topological vector space ${\rm NS}\, (S)_{\mathbf R}$. Our reference on basic facts on K3 surfaces and their projective models are \cite[Chapter VIII]{BHPV04} and \cite{SD74}.

{\bf Acknowledgement.} I would like to express my thanks to Doctors Sergey Galkin, Damiano Testa, Professros Igor Dolgachev, Jun-Muk Hwang, 
Yuri Prokhorov and De-Qi Zhang for valuable discussions relevant to this work. This note is grown up from my talk at the conference ``Essential dimension and Cremona groups" held at Chern Institute, June 2012. I would like to express my thanks to Professor Ivan Cheltsov for invitation and hospitality. 

\section{Preliminaries.}

The following two theorems (Theorems (\ref{classical}), (\ref{elliptic})) explain well why Gizatullin's question is interesting only for quartic K3 surfaces among smooth hypersurfaces. Theorem (\ref{classical}) is due to Matsumura 
and Monsky 
(\cite{MM63}) when $n \ge 2$ 
and Chang (\cite{Ch78}) when $n=1$. 

\begin{theorem}\label{classical}
Let $n$ and $d$ be positive integers and $X$ be a smooth hypersurface of degree $d$ in ${\mathbf P}^{n+1}$. If $(n, d) \not= (2, 4), (1,3)$, then 
${\rm Aut}\, (X)$ is derived from ${\rm Aut}\, ({\mathbf P}^{n+1})$. 
\end{theorem}

When $(n, d) = (1,3)$, $X = E$ is a smooth cubic curve in ${\mathbf P}^{2}$. 
The following theorem should be classical and well known:
\begin{theorem}\label{elliptic}
Let $E \subset {\mathbf P}^2$ be a smooth cubic curve. Then 
${\rm Aut}\, (E)$ is derived from ${\rm Bir}\, ({\mathbf P}^{2})$.
\end{theorem} 
Since an argument below will be used in our proof of Theorem (\ref{main1}), we sketch the proof. The point is that the same proof works with no modification for an elliptic curve $E/K$ defined over any field $K$ of characteristic $0$ with $E(K) \not= \emptyset$. 
\begin{proof} Choose $O \in E$ and regard $O$ as the unit element of the elliptic curve $E$. Then ${\rm Aut}\, (E)$ is the semi-direct product 
of $E$, the group of translations, and ${\rm Aut}(E, O)$ fixing $O$. Here ${\rm Aut}\, (E, O)$ is a finite cyclic group of order $d \in \{2, 4, 6 \}$. 
By linear change of the coordinate of ${\mathbf P}^{2}$, we can rewrite the equation of $E$ in the Weierstrass form:
$$y^2 = x^3 + px + q\,\, .$$
If ${\rm Aut}\,(E, O)$ is of order $6$ (resp. $4$), we can make $p = 0$ 
(resp. $q =0$). The generator of ${\rm Aut}\,(E, O)$ is 
$(x, y) \mapsto (x, -y)$ 
when $d = 2$, $(x, y) \mapsto (-x, \sqrt{-1}y)$ when $d = 4$ and $(x, y) \mapsto (\zeta_3 x, -y)$ when $d = 6$. Here $\zeta_3$ is the primitive third root of unity. So ${\rm Aut}(E, O)$ is derived from ${\rm Aut}\, ({\mathbf P}^{2})$. Let us consider the translation of $E$ by $(a, b) \in E$. The classical addition formula says that $(s', t') = (s, t) + (a, b)$ on $E$ if and only if 
$$s' = (\frac{t-b}{s-a})^2 - s - a\,\, ,\,\, t' = \frac{t-b}{s-a}(s'-a) + b\,\, .$$
Consider the rational map $\varphi_{(a, b)} : {\mathbf P}^2 \cdots \to {\mathbf P}^2$ defined 
by the same equations above. Then given $(s', t') \in {\mathbf P}^{2}$, the term $(t-b)/(s-a)$ is uniquely determined from $(s', t')$ as a rational function of $(s', t')$, by the second formula. Then $s$ is uniquely determined by the first formula and $t$ is uniquely recovered from $(t-b)/(s-a)$, as rational functions of $(s', t')$. Thus $\varphi_{(a, b)} \in {\rm Bir}\, ({\mathbf P}^{2})$ and $\varphi_{(a, b)} \vert E$ is the original addition by 
$(a, b) \in E$. 
\end{proof}
One can also check that there are only finitely many $g \in {\rm Aut}\, (E)$ of the form $g = \tilde{g} \vert E$ with $\tilde{g} \in {\rm Aut}\, ({\mathbf P}^2)$. So, it is really necessary to enlarge ${\rm Aut}\, ({\mathbf P}^2)$ 
to ${\rm Bir}\, ({\mathbf P}^2)$ in Theorem (\ref{elliptic}). 

We close this section by recalling the following well known result 
which will be used in the remaining sections:

\begin{proposition}\label{k3}
Let $S$ be a K3 surface. Then:
\begin{enumerate}
\item Any embedding $\Phi : S \to {\mathbf P}^3$ 
is of the form $\Phi = \Phi_{\vert H \vert}$. Here $\Phi_{\vert H \vert}$ 
is the morphism associated with the complete linear system of a very ample divisor $H$ such that $(H^2)_S = 4$. 
\item Under (1), the group $\{g \in {\rm Aut}\, (S)\, \vert\, g^*H = H\}$ is the subgroup of ${\rm Aut}\, (S)$ consisting of elements being derived from ${\rm Aut}\, ({\mathbf P}^3)$ under $\Phi_{\vert H \vert}$ 
and this is a finite group. 
\item Let $g \in {\rm Aut}\, (S)$ and $H$ be a very ample line bundle such that $(H^2)_S = 4$. Then $g^*H' = H$ if and only if there is an isomorphism 
$\tilde{g} : {\mathbf P}^3 \to {\mathbf P}^3$ between the target spaces 
such that $\tilde{g} \circ \Phi_{\vert H \vert} = 
\Phi_{\vert H' \vert} \circ g$. 
\end{enumerate}
\end{proposition}
\begin{proof} If $S \subset {\mathbf P}^3$, then $S$ is of degree $4$ and $H^{0}({\mathbf P}^3, {\mathcal O}_{{\mathbf P}^3}(1)) \simeq H^{0}(S, {\mathcal O}_{S}(1))$ under the restriction map by $H^1({\mathbf P}^3, {\mathcal O}_{{\mathbf P}^3}(n)) =0$. We also recall that $S$ has no global holomorphic vector field other than $0$. All the results follow from these three facts.
\end{proof}
\section{Proof of Theorem (\ref{main1}).}

Let 
$$S_0 := (x_0^4 -x_1^4 + x_2^4 - x_3^4 = 0) \subset {\mathbf P}^3\,\, .$$
$S_0$ is isomorphic to the Fermat quartic K3 surface and contains skew lines
$$L := (x_0 = x_1\,\, ,\,\, x_2 = x_3)\,\, ,\,\, 
M := (x_0 = - x_1\,\, , \,\, x_2 = -x_3)\,\, ,\,\, L \cap M = \emptyset\,\, .$$
As observed by the intersection matrix below, the sublattice 
${\mathbf Z}\langle [H]. [L], [M] \rangle$, where and hereafter $H$ is the hyperplane class, is primitive in $H^2(S_0, {\mathbf Z})$. 
Let $S$ be a small {\it generic} deformation of $S_0$ inside ${\mathbf P}^3$ 
such that $L, M \subset S$. Then $S \subset {\mathbf P}^3$ is a smooth quartic K3 surface containing $L$ and $M$. In Hodge theoretic terms, $S$ is a generic small deformation of $S_0$ keeping the classes $[H]$, $[L]$, $[M]$ being $(1,1)$-classes. Thus, such $S$ form a dense subset of a $17$ dimensional family, and by the primitivity mentioned above, ${\rm NS}\, (S)  = {\mathbf Z}H \oplus {\mathbf Z}L \oplus {\mathbf Z}M$. 
The intersection matrix is:
$$\left(\begin{array}{rrr}
(H^2)_S & (H.L)_S & (H.M)_S\\
(L.H)_S & (L^2)_S & (L.M)_S\\
(M.H)_S & (M.L)_S & (M^2)_S 
\end{array} \right)\,\, = \,\, \left(\begin{array}{rrr}
4 & 1 & 1\\
1 & -2 & 0\\
1 & 0 & -2 
\end{array} \right)\,\, .$$

We are goning to prove that this $S$ satisfies all the requirements of Theorem 
(\ref{main1}). In what follows, we set:
$$f := H - L\,\, ,\,\, e := H-L+M\,\, ,\,\, v := -6H + 7L -3M\,\, ,$$
$$f' := H - M\,\, ,\,\, e' := H-M+L\,\, ,\,\, v' := -6H -3L +7M\,\, ,$$ 
in ${\rm NS}\, (S)$. 

\begin{lemma}\label{lem31}
$S$ satisfies:

\begin{enumerate}
\item ${\rm NS}\, (S) = {\mathbf Z}f \oplus {\mathbf Z}f \oplus {\mathbf Z}v$.
\item $(e^2)_S = (f^2)_S = 0$, $(e,f)_S = 1$, $(v, f)_S = (v, e)_S = 0$ and $(v^2)_S = -20$.  
\item ${\rm NS}\, (S)^*/{\rm NS}\, (S) = \langle \overline{v/20} \rangle \simeq {\mathbf Z}_{20}$, 
where $\overline{x} = x\, {\rm mod}\, {\rm NS}\, (S)$ for $x \in {\rm NS}\, (S)_{{\mathbf Q}}$. 
\end{enumerate} 
The same are true for $f'$, $e'$, $v'$. 
\end{lemma}
\begin{proof} (2) follows from explicit calculation based on the intersection matrix of $H$, $L$, $M$. We have ${\mathbf Z}\langle e, f, v \rangle \subset {\mathbf Z}\langle H, L, M \rangle$ by definition. Agan by explicit calculaion, we see that the determinant of the intersection matrix of $H, L, M$ and $f, e, v$ are both $-20$. This implies (1). Since the determinant of ${\rm NS}\, (S)$ is 
$-20$, we have $\vert {\rm NS}\, (S)^*/{\rm NS}\, (S) \vert = 20$ by elementary divisor theory. By (2), $\overline{v/20} \in {\rm NS}\, (S)^*/{\rm NS}\, (S)$ and it is of order $20$. This implies (3). The proof for $f'$, $e'$, $v'$ is identical.
\end{proof}

Let $* : {\rm Aut}\, (S) \to {\rm O}\, ({\rm NS}\, (S))$ be the natural contravariant group homomorphism.

\begin{lemma}\label{lem32}
$S$ satisfies:

\begin{enumerate}
\item Let $g \in {\rm Aut}\, (S)$. Then $g^*\sigma_S = \pm \sigma_S$. Moreover 
$g^*\sigma_S = \sigma_S$ (resp. $g^*\sigma_S = -\sigma_S$) if and only if $g^*\vert {\rm NS}\,(S)^*/{\rm NS}\, (S) = id$ (resp. $g^*\vert {\rm NS}\,(S)^*/{\rm NS}\, (S) = -id$).
\item The map $*$ is injective.  
\end{enumerate} 
\end{lemma}

\begin{proof} Since $S$ is projective and ${\rm rank}\, T(S) = 19$ is odd, it follows that $g^*\sigma_S = \pm \sigma_S$ by Nikulin (\cite{Ni80}). Thus $g^* \vert T(S) = \pm id$ (\cite{Ni80}). Since there are only two cases above and $T(S)^*/T(S) = {\mathbf Z}_{20}$, it follows that $g^*\sigma_S = \pm \sigma_S$ is equivalent to $g^* \vert T(S)^*/T(S) = \pm id$ respectively. This implies (1) via the natural isomorphism explained in Notation. 
If $g^* = id$ on ${\rm NS}\, (S)$, then $g^*\vert {\rm NS}\, (S)^*/{\rm NS}\, (S) = id$. Thus $g^{*}\vert T(S) = id$ by (1). Hence 
$g^* \vert H^2(S, {\mathbf Z}) = id$. The result (2) now follows from 
the golobal Torelli theorem for K3 surfaces. 
\end{proof}

Planes $P$ such that $L \subset P \subset {\mathbf P}^3$ form
a linear pencil $\{P_t \vert t \in {\mathbf P}^1\}$. Put $h_t := P_t \vert S \in \vert H \vert$. 
Then $P_t \cap S = L \cup E_t$ where $E_t \subset P_t \simeq {\mathbf P}^2$ is 
a plane cubic curve such that $(E_t^2)_S = 0$ and $(E_t.M)_S = 1$.  Hence $\{ E_t\, \vert \, t \in {\mathbf P}^1 \}$ defines an elliptic 
fibration on $S$ with section $M$: 
$$\Phi_1 := \Phi_{\vert H - L \vert} : S \to {\mathbf P}^1\,\, .$$
We regard $M$ as the zero section. 
We denote by ${\rm Aut}\, (\Phi_1)$ the subgroup of ${\rm Aut}\, (S)$ which preserves the fibration $\Phi_1$, i.e., the set of $g \in {\rm Aut}\, (S)$ for which there is $\overline{g} \in {\rm Aut}\, ({\mathbf P}^1)$ such that $\Phi_1 \circ g = \overline{g} \circ \Phi_1$. Let $\iota_1$ be the inversion of $\Phi_1$ and ${\rm MW}(\Phi_1)$ be the Mordell-Weil group of $\Phi_1$, i.e., the group of translations by sections of $\Phi_1$.
Then ${\rm MW}(\Phi_1) \subset {\rm Aut}\, (\Phi_1)$ and $\iota_1 \in {\rm Aut}\, (\Phi_1)$. 

Similarly, we have an elliptic fibration  
$$\Phi_2 := \Phi_{\vert H - M \vert} : S \to {\mathbf P}^1\,\, ,$$
with section $L$. Regarding $L$ as the zero section of $\Phi_2$, we define ${\rm Aut}\, (\Phi_2)$, ${\rm MW}(\Phi_2)$ and the inversion $\iota_2$ of $\Phi_2$ similarly.

\begin{proposition}\label{lem33}
$S$ satisfies:

\begin{enumerate}
\item $\Phi_1$ has no reducible fiber.
\item
The set of global sections of $\Phi_1$ (in ${\rm NS}\, (S)$) is
$$\{10n^2 f + e + nv\, \vert\, n \in {\mathbf Z}\}\,\, .$$
\item
With respect to the basis $\langle f, e, v\rangle$ of 
${\rm NS}\, (S)$, 
$$\varphi_{n}^* = \left(\begin{array}{rrr}
1 & 10n^2 & 20n\\
0 & 1 & 0\\
0 & n & 1 
\end{array} \right)\,\, ,\,\, \iota_1^* = \,\, \left(\begin{array}{rrr}
1 & 0 & 0\\
0 & 1 & 0\\
0 & 0 & -1 
\end{array} \right)\,\, ,$$
where $\varphi_n^{-1} \in {\rm MW}\, (\Phi_1)$ is the element defined by the translation by $C_n := 10n^2 f + e + nv$. 
In particular, ${\rm MW}\, (\Phi_1) = \langle \varphi_1 \rangle 
\simeq {\mathbf Z}$. 
\end{enumerate}

The same are true for $\Phi_2$ if we replace the basis $\langle f, e, v\rangle$ of ${\rm NS}\, (S)$ to the basis $\langle f', e', v'\rangle$ of 
${\rm NS}\, (S)$. 
\end{proposition}

\begin{proof} $f$ is the class of fiber and $M$ is the zero section of $\Phi_1$. Observe that there is no $x \in {\rm NS}\, (S)$ such that 
$$(M.x)_S = (f.x)_S = 0\,\, ,\,\, (x^2)_S = -2\,\, .$$
In fact, writing $x = af + be + cv$, where $a, b, c \in {\mathbf Z}$, and 
substituting this and $e = M +f$ into the first two equations, we obtain that 
$a = b = 0$. Hence  $x = cv$. Then the last condition implies 
$c = 1/\sqrt{10} \not\in {\mathbf Z}$. Thus, there are no ${\mathbf P}^1$ in fibers of $\Phi_1$ and the result (1) follows.

Let $C$ be a section of $\Phi_1$. Then 
$$C = xf + ye + nv\,\, --- (*)$$ in ${\rm NS}\,(S)$ 
for some $x, y, n \in {\mathbf Z}$. Since $f$ is the class of fiber of $\Phi_1$, it follows that $(C.f)_S = 1$. Hence $y = 1$. Substituting into (*), we obtain $C = xf + e + nv$. Since $C \simeq {\mathbf P}^1$, it follows that $(C^2)_S = -2$, i.e., $2x -2 - 20n^2 = -2$. Hence $x = 10n^2$. Therefore $C = 10n^2 f + e + nv$ for some $n \in {\mathbf Z}$. Conversely, Let 
$c = 10n^2 f + e + nv \in {\rm NS}\, (S)$ with $n \in {\mathbf Z}$. We have 
$(c^2)_S = -2$. Hence either $c$ or $-c$ is represented by an effective curve. 
Since $(c.f)_S = 1$, it is $c$. Set $[C_0 + A] = c$, where $C_0$ and $A$ are effective curves (possibly $0$) such that $C_0$ is irreducible and $(C_0.f)_S \not= 0$. Since $(c.f)_S = 1$ and $f$ is nef, it follows that $(C_0.f)_S = 1$ and $(A.f)_S = 0$. Since $C_0$ is irreducible, it folows that $C_0$ is a section of $\Phi_1$. Moreover, $A$ is in fibers by $(A.f)_S = 0$. Since $\Phi_1$ has no reducible fiber, it follows that $A = mf$ in ${\rm NS}\, (S)$ for some $m \in {\mathbf Z}$. We calculate that
$$-2 = (c^2)_S = (C_0^2)_S + 2m(C_0.f)_S = -2 + 2m\,\, .$$
Hence $m=0$. Thus $A = 0$ as divisors. Hence $c = [C_0]$. This proves (2).

Let us show (3). We have $\iota_1^*f = f$ and $\iota_1^*e = e$. Hence 
$\iota_1^*v = -v$ by $\iota_1 \not= id_S$ and Lemma (\ref{lem32}). We have 
$$\varphi_n^*(f) = f\,\, ,\,\, \varphi_n^*(e) = C_n = 10n^2 f + e + nv\,\, .$$
Put $\varphi_n^*(v) = xf + ye + zv$ where $x, y, z \in {\mathbf Z}$. 
By substituting this into
$$(\varphi_n^*(f).\varphi_n^*(v))_S = (f.v)_S = 0\,\, ,\,\,  (\varphi_n^*(e).\varphi_n^*(v))_S = (e.v)_S = 0\,\, ,\,\, (\varphi_n^*(v)^2)_S = (v^2)_S = -20\,\, ,$$
we obtain $\varphi_n^*(v) = \pm (20nf +v)$. Since $\varphi_n^*\sigma_S = \sigma_S$, it follows that $\varphi_n^* \vert {\rm NS}\, (S)^*/{\rm NS}\,(S) = id$ 
by Lemma (\ref{lem32})(1). Thus $\varphi_n^*(v) = 20nf +v$. 
By the explicit form of the matrix and using induction on $\pm n$, one can see that $\varphi_n = (\varphi_1)^n$. This proves (3).

The proof for $\Phi_2$ is identical.
\end{proof}

\begin{proposition}\label{lem34}
For each $i=1$, $2$, we have:
$${\rm Aut}\, (\Phi_i) = {\rm MW}\, (\Phi_i) \cdot \langle \iota_i \rangle 
\simeq {\mathbf Z} \cdot {\mathbf Z}_2\,\, .$$
Here and hereafter for groups $A$ and $B$, the group $A \cdot B$ is the semi-direct product in which $A$ is normal. 
\end{proposition}

\begin{proof} We only prove for $i=1$. Proof for $i=2$ is identical.
We already observed that $\langle {\rm MW}(\Phi_1), \iota_1 \rangle \subset {\rm Aut}\, (\Phi_1)$. Let $g \in {\rm Aut}\, (\Phi_1)$. We want to show that $g  \in \langle {\rm MW}(\Phi_1), \iota_1 \rangle$. We have $g^*(H-L) = H-L$. By composing ${\rm MW}(\Phi_1)$, 
we may and will assume that $g^*M = M$. Hence
$$g^*e = e\,\, ,\,\, g^*f = f\,\, .$$ 
Thus $g^*v = \pm v$. As we observed, $\iota_1^*e = e$, $\iota_1^*f = f$ and 
$\iota_1^*v = -v$. Thus by composing $\iota_1$ if necessary, we may and will assume that $g^*v = v$. Then $g = id_S$ by Lemma (\ref{lem32})(2). Thus ${\rm Aut}\, (\Phi_1) = \langle {\rm MW}(\Phi_1), \iota_1 \rangle$. If $g \in {\rm MW}(\Phi_1)$, then $\iota_1 \circ g \circ \iota_1^{-1} \in {\rm MW}(\Phi_1)$. In fact, if $g$ is the translation defined by a section $C$, then $\iota_1 \circ g \circ \iota_1^{-1}$ is the translation defined by $\iota_1(C)$. Thus 
$g \in {\rm MW}(\Phi_1) \cdot \langle
\iota_1 \rangle$. Since $\iota_1$ is the inversion, we have $\langle \iota_1 \rangle \simeq {\mathbf Z}_2$ and ${\rm MW}(\Phi_1) \simeq {\mathbf Z}$ by 
Proposition (\ref{lem33})(2). This completes the proof.
\end{proof}

From now on, we set $\iota_3 := \iota_1 \circ \varphi_1$. 

\begin{corollary}\label{lem35}
With respect to the basis $\langle f, e, v\rangle$ of 
${\rm NS}\, (S)$, 
$$\iota_3^* = \left(\begin{array}{rrr}
1 & 10 & -20\\
0 & 1 & 0\\
0 & 1 & -1 
\end{array} \right)\,\, .$$
In particular, $\iota_3$ is also an involution.
\end{corollary}
\begin{proof} Since $\iota_3^* = (\iota_1 \circ \varphi_1)^* = \varphi_1^* \circ \iota_1^*$, the result follows from Proposition (\ref{lem33})(3). The fact that 
$\iota_3$ is an involution also follows from the semi-direct product structure.
\end{proof}

\begin{corollary}\label{lem36}
$\iota_1^*$, $\iota_2^*$ and $\iota_3^*$ are (orthogonal) reflections with respect to the hyperplanes orthgonal to the $-20$-elements, $-6H + 7L -3M$, $-6H -3L +7M$ and $4H -3L -3M$ respectively.
\end{corollary}
\begin{proof}
The results are clear for $\iota_1$, $\iota_2$ by their explicit matrix forms. 
Observe that $\iota_3^*$ is an orthogonal involution with determinant $-1$, being not $-id$. Hence $\iota_3^*$ is the reflection with respect to the hyperplane orthogonal to a primitive eigenvector, which is unique up to $\pm$, with eigenvalue $-1$ of $\iota_3^*$. Indeed, $4H -3L -3M$ is such a vector. 
\end{proof} 

The goal of this section is to prove the following:

\begin{theorem}\label{main3}
Let $S \subset {\mathbf P}^3$ be as above. Then:
\begin{enumerate}
\item ${\rm Aut}\, (S) = \langle \iota_1, \iota_2, \iota_3 \rangle \simeq {\mathbf Z}_2 * {\mathbf Z}_2 * {\mathbf Z}_2$. 
\item No element other than $id_S$ is derived from 
${\rm Aut}\, ({\mathbf P}^3)$.
\item Every element of ${\rm Aut}\, (S)$ is derived from ${\rm Bir}\, ({\mathbf P}^3)$.
\item $S$ admits infinitely many different embeddings $S \to {\mathbf P}^3$, up to ${\rm Aut}\, ({\mathbf P}^3)$, and ${\rm Aut}\, (S)$ is derived from ${\rm Bir}\, ({\mathbf P}^3)$ in any embedding $S \to {\mathbf P}^3$. 
\end{enumerate} 
\end{theorem}

We prove Theorem (\ref{main3}) by deviding into several lemmas. 

\begin{lemma}\label{lem37}
Let $g \in {\rm Aut}\, (S)$. If $g^*H = H$, then $g = id_S$. In particular, 
no element of ${\rm Aut}\, (S)$ other than $id_S$ is derived from ${\rm Aut}\, ({\mathbf P}^3)$.
\end{lemma}
\begin{proof}
Consider the following system of equations of $x, y, z \in {\mathbf Z}$:
$$((xH + yL +zM)^2)_S = -2\,\, ,\,\, ((xH + yL +zM).H)_S = 1\,\, .$$
The solutions are $(x, y, z) = (0,0, 1)$ and $(0,1,0)$. Thus if $g^*H = H$, 
then either $g^*L = L$ and $g^*M = M$ or $g^*L = M$ and $g^*L = M$. 
In the first case, $g = id_S$ by Lemma (\ref{lem32})(2). In the second case, 
$g^*(-6H +7L -3M) = -6H + 7L -3M$, where $v = -6H + 7L -3M$. Thus 
$g^*(v/20) \not= \pm v/20$ in ${\rm NS}\, (S)$, a contradiction to 
Lemma (\ref{lem32})(1). This proves the result.  
\end{proof}

\begin{lemma}\label{lem38}
The following cone $D$
$$D := {\mathbf R}_{\ge 0}(H + \frac{L}{2} + \frac{M}{2}) + {\mathbf R}_{\ge 0}(H - L + \frac{M}{2}) + {\mathbf R}_{\ge 0}(H + \frac{L}{2} - M) + {\mathbf R}_{\ge 0}(H - L) + {\mathbf R}_{\ge 0}(H - M)\,\, ,$$
that is, the cone over the pentagon of the five vertices as indicated, is a fundamenal domain for the action of $\langle \iota_1, \iota_2, \iota_3 \rangle$ on the effective nef cone $\overline{{\rm Amp}}^{e}\,(S)$ of $S$.  
\end{lemma}
{\it Readers are strongly encouraged to draw a picture, cut out by 
$(H.*)_S = 1$, by oneselves.} 
\begin{proof} The first three vertices are in the 
interior of the positive cone. The last two vertices are on the boundary of the positive cone, corresponding to $\Phi_1$ and $\Phi_2$. The first ray ${\mathbf R}_{\ge 0}(H + \frac{L}{2} + \frac{M}{2})$ gives the contraction 
of both $L$ and $M$. The face spanned by $H + \frac{L}{2} + \frac{M}{2}$ and $H - L + \frac{M}{2}$ corresponds to the contraction of $M$, and the face spanned by $H + \frac{L}{2} + \frac{M}{2}$ and $H + \frac{L}{2} -M$ corresponds to the contraction of $L$. Thus they are (part of the) faces of $\overline{{\rm Amp}}^{e}\,(S)$ for which $H$ is on the below (resp. left) side. Consider the remaining three faces. They are in the same side as $H$ with respect to the first two faces. The face spanned by $H - L + \frac{M}{2}$ and $H - L$ is orthogonal to $-6H + 7L -3M$, the face spanned by $H + \frac{L}{2} - M$ and $H - M$ is orthogonal to $-6H -3L +7M$ and the last face spanned by $H-L$ and $H-M$ is orthogonal to $4H -3L -3M$. Thus these three faces are (part of) the invariant hyperplanes of 
the orthogonal {\it reflections} $\iota_1$, $\iota_2$ and $\iota_3$ respectively. 
This implies the result. 
\end{proof}
\begin{lemma}\label{lem39}
The set $Q := \{x \in {\rm NS}\, (S)\, \vert\, 
x \in D\,\, , \,\, (x^2)_S = 4\}$,  
where $D$ is the cone in Lemma (\ref{lem38}), is
$$\{ H\,\, ,\,\, 2H-L-M\,\, ,\,\, 3H -3L +M\,\, ,\,\, 3H -3M +L\,\, .\}\,\, .$$
Moreover, among these four elements, only $H$ is very ample.
\end{lemma}
\begin{proof} Note that $3H -3L +M = 3f + M$. Thus $\vert 3H -3L + M \vert = \vert 3f \vert + M$. This is because $\dim \vert 3H -3L +M \vert = \dim \vert 3f \vert = 3$ by the Reiman-Roch formula. Similarly $\vert 3H -3M +L \vert = \vert 3(H-M) \vert + L$. $2H-L-M = (H-L) + (H-M)$ gives a double cover of ${\mathbf P}^1 \times {\mathbf P}^1$ given by $\Phi_1 \times \Phi_2$. This proves the last statement. Let $A \in Q$. Since $A \in {\rm NS}\, (S)$, we can write $A = xH + yL + zM$ for some $x, y, z \in {\mathbf Z}$. Since $A \in D$, it follows that $x > 0$. Set 
$$s :=  \frac{y}{x}\,\, ,\,\, t :=  \frac{z}{x}\,\, .$$
By symmetry of numerical conditions, we may and will assume that $s \le t$. 
Since $A \in D$ with $s \le t$, it follows that 
$$-1 \le s \le  t \le \frac{1}{2}\,\, ,\,\, -1 \le s + t\,\, .$$
Such $(s, t)$ form the quadrangle with four vertices $(1/2, 1/2)$, $(-1, 1/2)$, $(-1, 0)$, $(-1/2, -1/2)$ in the $st$-plane. 

{\it Here redarers are strongly recommended to draw a picture by oneselves.} 

On the other hand, by $(A^2)_S = 4$, we have 
$$4x^2 - 2y^2 -2z^2 +2xy + 2xz = 4\,\, .$$ 
Substituting $y = sx$ and $z = tx$, dividing by $x^2 \not= 0$ and completing squares, we obtain
$$(s - \frac{1}{2})^2 +(t - \frac{1}{2})^2 = 
\frac{5}{2} - \frac{2}{x^2}\,\, .$$
This is the circle with center $(1/2, 1/2)$ in the $st$-plane. 

Assume that $x = 1$. Then $s=y$ and $t=z$ are integers such that $(y - \frac{1}{2})^2 +(y - \frac{1}{2})^2 = \frac{1}{2}$. Solutions are only $(y,z) = (0,0)$ 
and $(1,1)$. Assume that $x = 2$. Then $2s=y$ and $2t=z$ are integers 
and satisfy $(y-1)^2 +(z-1)^2 = 8$. Partition of $8$ 
into two squares is only 
$8 = 4+4$. Hence, the solutions with $s \le t$ are 
only $(y,z) = (3, 3)$, $(-1, 3)$  
and $(-1,-1)$. But the first two solutions do not satisfies $t = z/2 \le 1$. Hence $(y,z) = (-1, -1)$. Assume that $x = 3$. Then $3s=y$ and $3t=z$ are integers and satisfy $(2y-3)^2 +(2z-3)^2 = 82$. Here we note that $2y-3$ and $2z-3$ are odd integers. Partition of $82$ into odd squares is only 
$82 = 1^2 + 9^2$ up to the order. By taking into accout that $-1 \le s \le t \le 1/2$, i.e., 
$-3 \le y \le z \le 3/2$, one can see that the solutions are only $(y,z) = (-3, 1)$. These four solutions give the four elemets in the statement. 

Assume that $x \ge 4$. One can compute the intersection points of the cricle above with the boundary $s+t = -1$ of the quadrangle above. In fact, substituting $t = -1-s$ into the eqaution of 
the circle above and symplifying it, 
we obtain 
$$s^2 + s = -\frac{1}{x^2}\,\, .$$
Solving this on $s$ in the range $-1 \le s \le t \le 1/2$ by using the root formula of quadratic equation, we find that the value of the $s$-coordinate at the intersection point is:
$$-\frac{1}{2} - \sqrt{\frac{1}{4} - \frac{1}{x^2}}\,\, .$$
From this calculation with $-1 \le s$, we obtain 
$$-1 \le s \le -\frac{1}{2} - \sqrt{\frac{1}{4} - \frac{1}{x^2}}\,\, .$$
If $s = -1$, then by substituting this into the equation of the circle above and solving it on $t$ by using the root formula of quadratic equation, we obtain
$$t = \frac{1}{2} \pm \frac{\sqrt{x^2 - 8}}{2x}\,\, .$$
Since $t$ is a rational number and $x$ is an integer, it follows that $x^2 -8$ 
is a square of some integer, i.e., $x^2 -4 = a^2$ for some non-negative integer $a$. This equation is equivalent to $(x-a)(x+a) = 8$. Here $x+a \ge 4$. Hence 
$x+a =4$ or $8$ and therefore $x - a =2$ or $1$ respectively. The solutions are $(x, a) = (3, 1)$ and $(9/2, 1/2)$. Since $x \ge 4$, $x = 9/2$. But this is not an integer, a contradiction. Hence $s \not= -1$, i.e., $s > -1$. Since $s$ is a rational number whose demoninator divides $x$, it follows from the inequality above for $s$ that 
$$-\frac{1}{2} - \sqrt{\frac{1}{4} - \frac{1}{x^2}} - (-1) \ge \frac{1}{x}\,\, ,$$
that is,
$$\frac{1}{2} - \frac{1}{x} \ge \sqrt{\frac{1}{4} - \frac{1}{x^2}}\,\, .$$
In particular, the left hand side is also non-negative. Thus
$$(\frac{1}{2} - \frac{1}{x})^2 \ge \frac{1}{4} - \frac{1}{x^2}\,\, .$$
Expanding the right hand side and transform the right hand side to the left, 
we obtain:  
$$\frac{2}{x^2} - \frac{1}{x} \ge 0\,\, .$$
Multplying both sides by $x^2 >0$, we obtain $2 - x \ge 0$. Hence $x \le 2$. However, this contradicts $x \ge 4$. Hence $x \le 3$. This completes the proof.
\end{proof}

\begin{lemma}\label{lem310}
${\rm Aut}\, (S) = \langle \iota_1, \iota_2, \iota_3 \rangle$. 
\end{lemma}

\begin{proof} We have $\langle \iota_1, \iota_2, \iota_3 \rangle \subset {\rm Aut}\, (S)$. Let $g \in {\rm Aut}\, (S)$. Let $H' := g^*H$. Then $H' \in \overline{{\rm Amp}}^e\,(S)$. Hence there is $\varphi \in \langle \iota_1, \iota_2, \iota_3 \rangle$ such that $H'' := \varphi^{*}H' \in D$ by Lemma (\ref{lem38}). Since $H$ is in ${\rm NS}\, (S)$, very ample and 
$(H^2)_S = 4$, so is $H''$. Hence $H'' = H$ by Lemma (\ref{lem39}). That is,  
$(g \circ \varphi)^{*}H = H$. Thus $g \circ \varphi = id_S$ by Lemma (\ref{lem37}). Hence $g \in \langle \iota_1, \iota_2, \iota_3 \rangle$. 
\end{proof}

\begin{lemma}\label{lem311}
$\langle \iota_1, \iota_2, \iota_3 \rangle = \langle \iota_1 \rangle * \langle \iota_2 \rangle * \langle \iota_3 \rangle \simeq {\mathbf Z}_2 * 
{\mathbf Z}_2 * {\mathbf Z}_2$. 
\end{lemma}

\begin{proof} The positive cone $P(S)$ is divided into the four open domains by the three invariant hyperplanes $L_1$, $L_2$, $L_3$ corresponding the reflections $\iota_1$, $\iota_2$, $\iota_3$ respectively as in the proof of 
Lemma (\ref{lem38}). 

{\it Here readers are again strongly recommended to draw a picture 
by oneselves.}
 
Let $S_1$ be the one of the two domains of $P(S)$ divided by $L_1$, being in the {\it opposite} side of $H$. Similarly we define $S_2$ and $S_3$. Then the 
intersection of any two of $S_1$, $S_2$. $S_3$ is empty. Moreover $\iota_i^*(S_j) \subset S_i$ whenever $j \not= i$ and $\iota_i^*(H) \in S_i$. Hence, the result follows. Indeed,  
$$\iota_{k_1}^*\iota_{k_2}^* \cdots \iota_{k_m}^* \not= id$$ 
for any $k_1 \not= k_2 \not= \cdots \not= k_m$ such that $k_i \in \{1,2,3\}$. 
This is because the image of $H$ by the left hand side is in $S_{k_1}$ but $id(H) = H$ by the right hand side is not. This proves the result.  
\end{proof}

\begin{lemma}\label{lem312}
Every element of ${\rm Aut}\, (S)$ is derived from 
${\rm Bir}\, ({\mathbf P}^3)$. 
\end{lemma}
\begin{proof}
$\Phi_1$ is the restriction of 
the linear projection  
$$\Phi_{\vert P - L \vert} : {\mathbf P}^3 \cdots\to {\mathbf P}^1$$
from $L$. Let $\eta = {\rm Spec}\, {\mathbf C}({\mathbf P}^1)$. The fiber ${\mathbf P}_{\eta}^2$ of $\Phi_{\vert P - L \vert}$ over $\eta$ is 
the projective plane defined over ${\mathbf C}({\mathbf P}^1)$ and the generic fiber $S_{\eta}$ of $\Phi_1 = \Phi_{\vert H - L \vert}$ is a smooth cubic curve in ${\mathbf P}_{\eta}^2$ with $O \in S_{\eta}({\mathbf C}({\mathbf P}^1))$, corresponding to the section $M$, defined over ${\mathbf C}({\mathbf P}^1)$. Thus, ${\rm MW}\, (\Phi_1)$ and $\iota_1$ are automorphisms of the smooth cubic curve $S_{\eta}$ over ${\mathbf C}({\mathbf P}^1)$. Therefore by Theorem (\ref{elliptic}), ${\rm MW}\, (\Phi_1)$ and $\iota_1$ are derived from ${\rm Bir}\, ({\mathbf P}_{\eta}^2)$ over ${\mathbf C}(t)$. Since ${\rm Bir}\, ({\mathbf P}_{\eta}^2)$ is a subgroup of ${\rm Bir}\, ({\mathbf P}^3)$, the result follows.  The same is true for the inversion $\iota_2$ of $\Phi_2$. This proves the result.
\end{proof}

\begin{lemma}\label{lem313}
$S$ admits infinitely many different embeddings $S \to {\mathbf P}^3$ up to 
${\rm Aut}\, ({\mathbf P}^3)$. 
Moreover, ${\rm Aut}\, (S)$ is derived from ${\rm Bir}\, ({\mathbf P}^3)$ 
in any embedding $\Phi : S \to {\mathbf P}^3$. 
\end{lemma}
\begin{proof}
Since ${\rm Aut}\, (S)$ is an infinite group by Lemma (\ref{lem311}), it follows that ${\rm Aut}\, (S)^* H$ is an infinite set. Indeed otherewise, the stabilzer subgroup of $H$ would be an infinite group, a contradiction to 
Proposition (\ref{k3})(2). This proves the first assertion. Let $A$ be a very ample line bundle on $S$ such that $(A^2)_S = 4$. Then, by Lemmas (\ref{lem38}), (\ref{lem39}), there is $g \in {\rm Aut}\, (S)$ such that $g^*A = H$. Hence the last assertion follows from Proposition (\ref{k3})(3). 
\end{proof}

This completes the proof of Theorem (\ref{main3}). Theorem (\ref{main1}) follows from Theorem (\ref{main3}).

\section{Proof of Theorem (\ref{main2}).}

Let $\ell$ be an integer such that $\ell > 5$ and $S_{\ell}$ be a {\it generic} K3 surface such that ${\rm NS}\, (S_{\ell}) = L$ 
where 
$$ L := {\mathbf Z}h_1 + {\mathbf Z}h_2\,\, ,\,\, 
((h_i.h_j)_{S_{\ell}}) = \left(\begin{array}{rr}
4 & 4\ell\\
4\ell & 4
\end{array} \right)\,\, .$$
Here {\it generic} means that $g^* \sigma_S = \pm \sigma_S$ for all $g \in {\rm Aut}\, (S)$. 
Since the lattice $L$ is even of signature 
$(1,1)$, such K3 surfaces $S_{\ell}$ exist (\cite[Corollary 2.9]{Mo84}), are projective and form dense 
subset, in the classical topology, of the $18$-dimensional family of the $L$-polarized K3 surfaces. Here we note that non-generic ones are in the countable union of hyperplanes in the period domain of $L$-polarized K3 surfaces (\cite[Theorem 3.1]{Ni80}).

In this section, we prove the following: 
\begin{theorem}\label{main4}
Let $S := S_{\ell}$ be as above. Then:
\begin{enumerate}
\item ${\rm Aut}\, (S) \simeq {\mathbf Z}$.
\item The set ${\mathcal Q}$ of embeddings $\Phi : S \to {\mathbf P}^3$, up to 
${\rm Aut}\, ({\mathbf P}^3)$, is an infinite set.  
\item For any $(\Phi : S \to {\mathbf P}^3) \in {\mathcal Q}$ and for any $g \in {\rm Aut}\, (S) \setminus \{ id_S \}$, there is no $\tilde{g} \in {\rm Bir}\, ({\mathbf P}^3)$ such that $g = \tilde{g} \vert S$ 
with respect to the embedding $\Phi$. 
\end{enumerate} 
\end{theorem}

By replacing $(h_1, h_2)$ by $(-h_1, -h_2)$ if necessary, we may and will assume that $h_1$ is in the positive cone $P(S)$.

\begin{lemma}\label{lem44}
$S$ satisfies:
\begin{enumerate}
\item 
${\rm NS}\, (S)$ represents neither $0$ nor $\pm 2$, i.e., $(d^2)_S \not= 0, \pm 2$ for $d \in {\rm NS}\, (S)$. 
\item
$${\rm NS}\,(S)^*/{\rm NS}\, (S) = \langle \frac{h_1}{4}, \frac{h_2 - \ell h_1}{4(\ell^2 -1)} \rangle \simeq {\mathbf Z}_{4} \oplus {\mathbf Z}_{4(\ell^2 -1)}
\,\, .$$
\end{enumerate} 
\end{lemma}

\begin{proof} For $xh_1 + yh_2 \in {\rm NS}\, (S)_{\mathbf R}$, observe that
$$((xh_1 + yh_2)^2)_S = 4x^2 + 8\ell xy + 4y^2 
= 4((x+\ell y)^2 - (\ell^2 -1)y^2)\,\, .$$
The first equality shows that $(d^2)_S$ is divisible by $4$ for any $d \in {\rm NS}\, (S)$. Hence ${\rm NS}\, (S)$ does not represent $\pm 2$. Since $\sqrt{\ell^2 -1}$ is irrational by $\ell \ge 2$, it follows that $(d^2)_S \not= 0$ for any $d \in {\rm NS}\, (S)$. This proves (1). 
The assertion (2) follows from an explicit computation based on the elementary divisor theorem. Indeed, by performing elementary transformations for the matrix $((h_i.h_j)_S)$ in ${\mathbf Z}$, one can transform $((h_i.h_j)_S)$ to the 
diagonal matrix ${\rm diag}\, (4, 4(\ell^2 -1))$ such that 
$4 \vert 4(\ell^2 -1)$. This process also identifies the generators of ${\rm NS}\,(S)^*/{\rm NS}\, (S)$ as described. 
\end{proof}

In what follows, we set:
$$v_1 := (-\ell + \sqrt{\ell^2 -1})h_1 + h_2\,\, ,\,\, v_2 := h_1 + (-\ell + \sqrt{\ell^2 -1})h_2\,\, .$$
Note that $v_1$ and $v_2$ are both irrational. 

\begin{lemma}\label{lem45}
$$\overline{{\rm Amp}}\, (S) = {\overline P}\, (S) = {\mathbf R}_{\ge 0}v_1 + {\mathbf R}_{\ge 0}v_2\,\, .$$ 
\end{lemma}

\begin{proof}
Since $h_1$ is in the positive cone, the second equality follows from the second equality in the proof of Lemma (\ref{lem44})(1). Since ${\rm NS}\, (S)$ does not represent $-2$ by Lemma (\ref{lem44})(1), $S$ contains no ${\mathbf P}^1$. 
This implies the first equality. 
\end{proof}

\begin{lemma}\label{lem46}
$S$ satifies:
\begin{enumerate}
\item Let $g \in {\rm Aut}\, (S)$ such that $g^* \vert {\rm NS}\, (S)$ is of finite order. Then $g$ is of finite order.
\item ${\rm Aut}\, (S)$ has no element of finite order other than $id_S$.
\item For each $g \in {\rm Aut}\, (S)$, there are positive real numbers 
$\alpha(g)$, $\beta(g)$ such that $g^*v_1 = \alpha(g) v_1$ and 
$g^*v_2 = \beta(g) v_2$.
\item
${\rm Aut}\, (S) = \langle g_0 \rangle \simeq {\mathbf Z}.$
\end{enumerate} 
\end{lemma}

\begin{proof}
Let $g \in {\rm Aut}\, (S)$. Then either $g^*v_1 = \alpha v_1$ and $g^*v_2 = \beta v_2$ or $g^*v_1 = \alpha v_2$ and $g^*v_2 = \beta v_1$ for some real positive numbers $\alpha$ and $\beta$. 

The assertion (1) follows from $g^*\vert T(S) = \pm id$ and the global Torelli 
theorem for K3 surfaces. 

Let us prove (2). 
Assume that $g$ is of finite order. Then, in the first case, $\alpha = \beta = 1$, whence $g^* \vert {\rm NS}\, (S) = id$. Then $g^* \vert {\rm NS}\, (S)^*/{\rm NS}\, (S) = id$. On the other hand, $g^* \vert T(S) = \pm id$, by our genericity assumption. It follows that $g^* \vert T(S) = id$. Hence $g = id_S$ by the global Torelli theorem for K3 surfaces. In the second case, $(g^2)^*v_1 = \alpha^2 v_1$ and $(g^2)^*v_2 = \beta^2 v_2$. 
Hence $g^2 = id$ as we have shown. {\it Assume to the contrary that} $g \not= id_S$. If $g^*\sigma_S = \sigma_S$, then $g$ would have exactly $8$ fixed points (\cite[Section 5]{Ni80}). Thus ${\rm tr}\, (g^* \vert {\rm NS}\, (S)) = -14$ by 
the Lefschetz fixed point formula. However, this is impossible, bescause 
${\rm rank}\, {\rm NS}\, (S) = 2$ and $g^{*} \vert {\rm NS}\, (S)$ is of order $2$, so that the eigenvalues of $g^{*} \vert {\rm NS}\, (S)$ is in $\{\pm 1\}$ with multiplicity at most $2$. Consider the case $g^*\sigma_S = -\sigma_S$. Since $g$ is of finite order, there is an ample class $h$ such that $g^*h = h$. We can choose $h$ to be primitive. Then by $\langle h \rangle^*/\langle h \rangle \simeq (\langle h \rangle^{\perp})^*/\langle h \rangle^{\perp}$ and our case assumption, 
we would have $id = -id$ on $\langle h \rangle^*/\langle h \rangle$. Hence $(h^2)_S = 2$, a contradiction to Lemma (\ref{lem44})(1). 
This proves the assertion (2). 

Let us show (3). {\it Assume to the contrary that} $g^*v_1 = \alpha v_2$ and $g^*v_2 = \beta v_1$. Then, $g \not= id_S$ and the characteristic polynomial of $g^* \vert {\rm NS}\, (S)$ would be $t^2 - \alpha \beta \in {\mathbf Z}[t]$. Since $g^* \vert {\rm NS}\, (S) \in {\rm O}({\rm NS}\, (S))$ and $\alpha, \beta >0$, it would follow that 
$\alpha \beta = 1$. Hence $g^* \vert {\rm NS}\, (S)$ would of finite order, whence so would be $g$ by (1), a contradiction to (2). 

Let us show (4). By Sterk (\cite[Section 2]{St85}), the action ${\rm Aut}\, (S)^{*}$ 
on $\overline{{\rm Amp}}^e\, (S)$ has a finite rational polyhedral fundamental domain. Since $\overline{{\rm Amp}}^e\, (S)$ is irrational, it follows that 
$\vert {\rm Aut}\, (S) \vert = \infty$. Let ${\mathbf R}_{> 0}$ be the multiplicative group of positive real numbers. By (3), we have a well-defined group homomorphism 
$$\alpha : {\rm Aut}\, (S) \to {\mathbf R}^{> 0}\,\, ,\,\, g \mapsto \alpha(g)\,\, ,$$
where $g^*v_1 = \alpha(g)v_1$. If $\alpha(g) = 1$, then $\beta(g) = 1$ by $g^* \vert {\rm NS}\, (S) \in {\rm O}({\rm NS}\, (S))$ and $\alpha, \beta >0$. Then 
$g^* \vert {\rm NS}\, (S) = id$, whence $g = id_S$ by (1) and (2). Thus $\alpha$ is injective. Let 
$${\mathcal S} :=  \{ \alpha(g)\, \vert\,  g \in {\rm Aut}\, (S)\, , 
\alpha(g) > 1\}\,\, .$$
Since $\vert {\rm Aut}\, (S) \vert = \infty$, $\alpha$ is injective,  
and since $\alpha(g^{-1}) > 1$ if $\alpha(g) <1$, it follows that 
${\mathcal S}$ 
is not empty. Note that $\alpha(g)$ and $\beta(g)$ are zeros of quadratic equation of the form $t^2 - at + 1 = 0$ with $a \in {\mathbf Z}$. By the root formula of quadratic equation, we have $\alpha_{a_1} < \alpha_{a_2}$. Here $\alpha_{a_i}$ is the largest solution of $t^2 - a_i t + 1 = 0$ and $a_1$ and $a_2$ are integers such that $3 \le a_1 < a_2$. It follows that the set ${\mathcal S}$ has the minimum, say $\alpha_0 = \alpha(g_0) > 1$. 
Then for all $h \in {\rm Aut}\, (S)$, we have $h = g_0^n$ for some $n \in {\mathbf Z}$. In fact, we can take an integer $n$ such that $1 \le \alpha(h)\alpha(g_0)^{-n} < \alpha_0$. Then $\alpha(hg_0^{-n}) = 1$ by definition of $\alpha_0$. Since $\alpha$ is injective, it follows that $h = g_0^n$. Hence ${\rm Aut}\, (S) = \langle g_0 \rangle \simeq {\mathbf Z}$ as claimed. 
\end{proof}
In what follows, 
$${\mathcal H} := \{h \in \overline{{\rm Amp}}\, (S) \cap {\rm NS}\, (S)\, 
\vert\, (h^2)_S = 4\}\,\, ,$$
and $g_0$ is a generator of ${\rm Aut}\, (S)$ as in Lemma (\ref{lem46}) (4).

\begin{lemma}\label{lem47}
$S$ satisfies:
\begin{enumerate}
\item
Any $h \in {\mathcal H}$ is very ample. 
\item $\vert {\mathcal H} \vert = \infty$. 
\end{enumerate}
\end{lemma}

\begin{proof}
Since $S$ contains no ${\mathbf P}^1$, the complete linear system $\vert h \vert$ is free (\cite[2.7]{SD74}). Since there is no $d \in {\rm NS}\, (S)$ with 
$(d^2)_S \in \{0, \pm -2 \}$ by Lemma (\ref{lem44})(1), it follows 
that $\Phi_{\vert h \vert}$ is an embedding (\cite[Theorem 5.2]{SD74}). This proves (1). 
Note that $h_1 \in {\mathcal H}$ and ${\mathcal H} \not= \emptyset$. 
{\it Assume to the contrary that} $\vert {\mathcal H} \vert < \infty$. Then 
$\{(g_0^n)^*h_1\, \vert\, n \in {\mathbf Z}\}$ would be a finite set as well. Then there would be $m \in {\mathbf Z} \setminus \{0\}$ such that 
$(g_0^m)^* h_1 = h_1$. Since $h_1$ is ample, $g_0^m$ would be of finite order 
by Proposition (\ref{k3})(2), a contradiction to $\langle g_0 \rangle \simeq {\mathbf Z}$ (Lemma (\ref{lem46})(4)). Hence $\vert {\mathcal H} \vert = \infty$.
\end{proof}

\begin{lemma}\label{lem48}
Let $h \in {\mathcal H}$ and $\Phi : S \to {\mathbf P}^3$ be the embedding defined by the complete linear system $\vert h \vert$. We regard $S \subset {\mathbf P}^3$ by this $\Phi$. Let $H$ be the hyperplane class of ${\mathbf P}^3$ and $C$ be an effective curve on $S$ such that $(C.H)_{{\mathbf P}^3} < 16$. Then, there is a hypersurface $T$ in ${\mathbf P}^3$ such that $C = T \vert S$. 
\end{lemma}

\begin{proof} {\it Assuming to the contrary that there would be an effective curve $C \subset S$ such that the class $c := [C]$ is linearly independent to 
$h = H \vert S$ in ${\rm NS}\, (S)$, we shall derive a contradiction.} Then $N := \langle c, h \rangle$ would be a sublattice of ${\rm NS}\, (S)$ of the same rank $2$. In particular, the signature would be of $(1,1)$.
Then  
$$\vert N \vert := \vert {\rm det}\left(\begin{array}{rr}
(c^2)_S & (c.h)_S\\
(c.h)_S & (h^2)_S
\end{array} \right) \vert = (c.h)_S^2 - (c^2)_S \cdot (h^2)_S > 0\,\, ,$$
and $\vert N \vert$ would be divided by $\vert {\rm NS}(S) \vert := \vert {\rm det}(h_i.h_j) \vert$. 
Since $c = [C]$ is an effective class and $C \not\simeq {\mathbf P}^1$ 
by Lemma (\ref{lem44})(1), we have $(c^2)_S \ge 0$. Hence 
$$\vert N \vert \le (c.h)_S^2 < 16^2\,\, .$$
On the other hand, by $\ell > 5$, 
$$\vert {\rm NS}(S) \vert = 16\ell^2 - 16 > 
16 \cdot 17 - 16 = 16^2\,\, ,$$
and therefore $\vert N \vert$ could not be divided by 
$\vert {\rm NS}(S) \vert$, 
a contradiction. 
\end{proof}  

\begin{lemma}\label{lem49}
Let $h \in {\mathcal H}$ and $\Phi : S \to {\mathbf P}^3$ be the embedding defined by the complete linear system $\vert h \vert$. We regard $S \subset {\mathbf P}^3$ by this $\Phi$. Then, no element of 
${\rm Aut}\, (S) \setminus \{id_S \}$ is derived from 
${\rm Bir}\, ({\mathbf P}^3)$. 
\end{lemma}

\begin{proof} {\it Assuming to the contrary that there would be 
$g \in {\rm Aut}\, (S) \setminus \{id_S\}$ such that $g = {\tilde g} \vert S$ 
for some $\tilde{g} \in {\rm Bir}\, ({\mathbf P}^3)$, we shall derive a contradiction.} 

If ${\tilde g} \in {\rm Aut}\, ({\mathbf P}^3)$, then $g^*h = h$. Since $h$ is ample, $g$ would be of finite order by Proposition (\ref{k3})(2), 
a contradiction to Lemma (\ref{lem46})(4). 

If ${\tilde g} \in {\rm Bir}\, ({\mathbf P}^3) \setminus {\rm Aut}\, ({\mathbf P}^3)$, then by Lemma (\ref{lem48}), one can apply Theorem (\ref{takahashi}) for ${\tilde g}$. However, then $g(S) = g_*S \not= S$, a contradiction.
\end{proof}

This completes the proof of Theorem (\ref{main4}). Theorem (\ref{main2}) follows from Theorem (\ref{main4}).

\end{document}